\documentclass[12pt]{article}

\usepackage{amssymb, graphicx, parskip}

\usepackage{lipsum}
\usepackage[utf8]{inputenc}
\usepackage{titlesec}
\usepackage{enumitem}
\usepackage{amssymb,amsmath,amsthm}
\usepackage{graphicx}
\usepackage{todonotes}
\usepackage[utf8]{inputenc}
\usepackage[english]{babel}
\usepackage{float}
\usepackage{multibib}
\usepackage{appendix}
\usepackage[margin=2cm]{geometry}
\usepackage{hyperref}
\usepackage{todonotes}
\usepackage{ulem,cancel}

\makeatletter
\newcommand{\chapterauthor}[1]{%
  {\parindent0pt\vspace*{-25pt}%
  \linespread{1.1}\large\scshape#1%
  \par\nobreak\vspace*{35pt}}
  \@afterheading%
}
\makeatother

\numberwithin{equation}{section}
\newtheorem{theorem}{Theorem}[section]

\newtheorem{lemma}[theorem]{Lemma}

\theoremstyle{definition}
\newtheorem{remark}[theorem]{Remark}

\newcites{intro,urn,bp,ce,con}%
{References,
References,
References,
References,
References}

\usepackage{mathtools}
\newcommand{\R}{\right}

\newcommand{\F}{\left}

\newcommand{\E}{\mathbb{E}}

\newcommand{\bs}[1]{\boldsymbol{#1}}

\newcommand{\e}{\mathrm{e}}

\title{More on the asymptotic behaviour of moments of branching Markov processes}
\author{Christopher B. C. Dean\footnote{Department of Statistics, University of Warwick, Coventry, CV4 7AL, UK. Email: \texttt{\{Christopher.B.C.Dean\}, \{Emma.Horton\}@warwick.ac.uk}} \and Emma Horton$^*$}
\date{}

\begin{document}
\maketitle

\begin{abstract}
Consider a branching Markov process, $X = (X_t, t \ge 0)$, with non-local branching mechanism. Studying the asymptotic behaviour of the moments of $X$ has recently received attention in the literature \cite{genealogies_felix, bmoments} due to the importance of these results in understanding the underlying genealogical structure of $X$. In this article, we generalise the results of \cite{bmoments} to allow for a non-simple leading eigenvalue and to also study the higher order fluctuations of the moments of $X$. These results will be useful for proving central limit theorems and extending well-known LLN results. 

\medskip


\medskip

\end{abstract}

\section{Introduction}
Let $E$ be a Polish space and $\dagger \notin E$ a cemetery (absorbing) state. We consider a measure-valued stochastic process $X =(X_t, t\geq0)$ given by
\[
X_t := \sum_{i = 1}^{N_t}\delta_{x_i(t)}, \qquad t \ge 0, 
\]
whose atoms $\{x_i(t) : i = 1, \dots, N_t\}$ evolve in $E\cup \{\dagger\}$ according to the following dynamics. From an initial position $x \in E$, particles evolve independently in $E$ according to a Markov process $(\xi, \mathbf P_x)$. When at $y \in E$, particles branch at rate $\gamma(y)$, which we assume to be bounded, at which point the particle is replaced by a random number, $N$, of offspring at positions $x_1, \dots, x_N$ in $E$. The law of the offspring locations and their number is denoted by $\mathcal P_y$, where $y$ denotes the position of the parent particle at the branch time. We will also often use the notation $\mathcal E_y$ for the associated expectation operator. We refer to $X$ as a $(\mathbf P, \gamma, \mathcal P)$-branching Markov process (BMP). We will let $\mathbb P_{\mu}$ denote its law when initiated from $\mu \in \mathcal M(E)$, the space of finite measures on $E$. We also let $B(E)$ denote the space of bounded complex-valued functions, $B^+(E)$ denote the space of non-negative real bounded functions and $B_1(E)$, (resp. $B_1^+(E)$) denote those functions in $B(E)$, (resp. $B^+(E)$) that are bounded by unity. 

One way to characterise the behaviour of $X$ is via its linear expectation semigroup,
\begin{equation}\label{eq:linear}
  \psi_t[f](x) := \E_{\delta_x}\left[X_t[f]\right], \quad t \ge 0,\, x \in E, \, f \in B(E),
\end{equation}
where \(\E_{\delta_x}\) denotes the expectation conditional on \(X_0\) consisting of  a single particle at \(x\), and, for a measure $\mu$ and a function $f$, we have used the notation
\[
  \mu[f] := \int_E f(x)\mu({\rm d}x).
\]
In many cases, it is well known that the semigroup $\psi$ satisfies the following Perron Frobenius-type decomposition.
\begin{enumerate}[label=(PF),ref=(PF)]
\item \label{PF}
There exists $\lambda_1 \in \mathbb R$, $\varphi_1 \in B^+(E)$ and a probability measure $\tilde\varphi_1$ on $E$ such that 
\[
  \psi_t[\varphi_1] = {\rm e}^{\lambda_1 t}\varphi_1, \qquad \tilde\varphi_1[\psi_t[f]] = {\rm e}^{\lambda_1 t}\tilde\varphi_1[f],
\]
and
\[
\sup_{x \in E, f \in B_1(E)}\left| {\rm e}^{-\lambda_1 t}\psi_t[f](x) - \tilde\varphi_1[f]\varphi_1(x)\right| \to 0, \qquad \text{as } t \to \infty.
\]
\end{enumerate}
When this decomposition holds, along with moment conditions on the offspring distribution, it has been shown in \cite{bmoments} that similar asymptotics hold for the moments $\mathbb E_{\delta_x}[X_t[f]^k]$. We also mention the previous works \cite{durhamCMJ, meatman, iscoe, Klenke}, where asymptotics for moments of branching processes were studied under varying assumptions. Results of this kind are particularly useful in proving laws of large numbers \cite{AH1, EHK, bNTEbook}, for setting benchmarks for Monte Carlo codes \cite{NTEmoments} and for understanding the genealogical structure of branching Markov processes \cite{genealogies_felix, M2few, pedro}. However, these results do not give a complete description of the asymptotic behaviour of $X$. Indeed, they do not allow us to say anything about the fluctuations of the process via a CLT, for example. Moreover, there are many cases where \ref{PF} is not satisfied, for example when the first eigenvalue is not simple. The aim of this article is to study the asymptotic moments of $X$ in more depth by considering what happens when we have a more general version of \ref{PF}. These results will play a key role in proving CLTs for $X$ and will also allow us to provide new proofs for a LLN result.

\subsection{Assumptions}
{ Before stating our main results, we first spend some time discussing the assumptions. Our main assumption generalises \ref{PF} to the case when $\lambda_1$ is not a simple eigenvalue and where we may have more points in the discrete spectrum of $\psi$. Here and throughout, we will use the notation ${\rm ker}(\lambda_\ell) := \{f \in B(E):\tilde\varphi_{i,j}^{(k)}[f]=0, 1\leq i \leq \ell,1\leq j \leq p_i, 1 \leq k \leq k_{i,j}\}$ for each $1 \le \ell \le m$, where the \(\tilde\varphi_{i,j}^{(k)}\) are defined in the following assumption.

{\begin{enumerate}[label=(H1),ref=(H1)]
\item \label{H1b} There exists an integer \(m\geq 1\), a real number \(\lambda_1>0\), \(\lambda_{2},\dots ,\lambda_{m}\in \mathbb{C}\), that satisfy\footnote{If $m = 1$, the list $\lambda_2 \dots, \lambda_m$ is an empty list.} \(\lambda_1>\mathrm{Re}\lambda_2\geq \dots \geq \mathrm{Re}\lambda_m\), such that, for \(1\leq i \leq m\), there exist bounded functions $$\varphi_{i,1}^{(1)},\dots,\varphi_{i,1}^{(k_{i,1})},\dots,\varphi_{i,p_i}^{(1)},\dots,\varphi_{i,p_i}^{(k_{i,p_i})} \in B(E),$$ bounded linear functionals $$\tilde\varphi_{i,1}^{(1)},\dots,\tilde\varphi_{i,1}^{(k_{i,1})},\dots,\tilde\varphi_{i,p_i}^{(1)},\dots,\tilde\varphi_{i,p_i}^{(k_{i,p_i})}  :B(E) \rightarrow\mathbb{C},$$ such that $\tilde\varphi_{i,j}^{(k)}[\varphi_{i,j}^{(k)}] = 1$, \(\tilde\varphi_{i,j}^{(k)}[\varphi_{\ell,n}^{(q)}]=0\) for each \((i,j,k)\neq (\ell,n,q)\). There also exists and a linear operator \(\mathcal{N}\), such that
\begin{itemize}
    \item for each $1\leq i \leq m, \, 1\leq k \leq k_{i,1}$, \, $\mathcal{N}\varphi_{i,1}^{(k)} = 0$;
    \item for each $1\leq i\leq m, \, 2\leq j \leq p_i, \, 1\leq k \leq k_{i,j}$ there exists\footnote{If $k_1 \neq k_2$ then $k^*(i, j, k_1) \neq k^*(i, j, k_2)$} $k^* \in \{1, \dots, k_{i, j-1}\}$ such that $\mathcal{N}\varphi_{i,j}^{(k)} = \varphi_{i,j-1}^{(k^*)}$;
    \item for $f \in \mathrm{ker}(\lambda_m)$, $\mathcal{N}f = 0$,
\end{itemize}
and, for any $x \in E$, $t \ge 0$ and $f \in B(E)$,
\begin{equation}\label{eq:gen-eigenfunctions}
  \psi_t[\varphi_{i,j}^{(k)}](x) = {\rm e}^{(\lambda_i+\mathcal{N}) t}\varphi_{i,j}^{(k)}(x), \quad \text{and}\quad \tilde\varphi_{i,j}^{(k)}[\psi_t[f]] = {\rm e}^{\lambda_i t}\tilde\varphi_{i,j}^{(k)}[{\rm e}^{\mathcal{N}t}f].
\end{equation}
Moreover, 
\begin{equation}
    \label{eq: h1b assum}
  \sup_{x \in E, f \in B_1(E)}\e^{-\mathrm{Re}\lambda_m t}\bigg|\psi_t[f](x) - \sum_{j=1}^{p_i}\e^{(\lambda_i+\mathcal{N})t}\Phi_{i}[f](x)\bigg|
  \to 0, \quad \text{as } t \to \infty,
\end{equation}
where 
\[\Phi_{i,j}[f](x) = \sum_{k=1}^{k_{i,j}}\tilde\varphi_{i,j}^{(k)}[f]\varphi_{i,j}^{(k)}(x), \quad \Phi_{i}[f](x)=\sum_{j=1}^{p_i}\Phi_{i,j}[f](x).\]
\end{enumerate}}

Let us spend some time unpacking the above assumption. First note that, as mentioned above, this allows for the case that the leading eigenvalue associated with $\psi$ is not simple. For simplicity, assume that $m=1$ in \ref{H1b}. In the case that $p_1 = k_{1, 1} = 1$, we recover \ref{PF}. In the case that $p_1 = 1$ but $k_{1, 1} > 1$, we extend \ref{PF} to the case where $\lambda_1$ is not simple. This case implies that $k_{1, 1}$ is both the algebraic and geometric multiplicity of $\lambda_1$, with corresponding eigenfunctions given by $\varphi^{(1)}, \dots, \varphi^{(k_{1, 1})}$ (note, we have dropped the subscript for simplicity). The fact that these are eigenfunctions can be deduced from \eqref{eq:gen-eigenfunctions} and the first property of $\mathcal N$. The case that $p_1 = 2$, say, the functions $\varphi_{1, 1}^{(1)}, \dots, \varphi_{1, 1}^{(k_{1, 1})}$ are eigenfunctions, but the functions $\varphi_{1, 2}^{(1)}, \dots, \varphi_{1, 2}^{(k_{1, 2})}$ are generalised eigenfunctions of rank $2$, meaning $\mathcal N\varphi_{1, 2}^{(k)} \neq 0$ but $\mathcal N^2\varphi_{1, 2}^{(k)} = 0$, which can be deduced from the first two properties of $\mathcal N$. In general, if $p_1 \ge 3$, then $\varphi_{1, j}^{(k)}$ is of rank $j$. Note that, in the case that $m=1$ and $p_1 \ge 2$, the geometric multiplicity of $\lambda_1$ is $k_{1, 1}$ and the algebraic multiplicity is {$\sum_{j = 1}^{p_1} k_{1, j}$.} 
Allowing $m \ge 2$ generalises this to $m$ eigenvalues \(\e^{\lambda_1 t},\dots,\e^{\lambda_m t}\), each with eigenfunctions $\varphi_{i, 1}^{(1)}, \dots, \varphi_{i, 1}^{(k_{i, 1})}$ and generalised eigenfunctions $\varphi_{i, j}^{(1)}, \dots, \varphi_{i, j}^{(k_{i, j})}$ of rank $j$. The asymptotic given in \eqref{eq: h1b assum} stipulates that, when normalised by the smallest growth rate ${\rm e}^{\mathrm{Re}\lambda_m t}$, the expectation semigroup converges.

\subsection{Notation}
We are interested in seeing how \ref{H1b} propagates to higher moments of $X_t$. We will shortly see that it will obey one of three possible dynamics depending on the growth rate of \(\psi_t[f]\). However, we first introduce some notation that will be used throughout the rest of the article. 

For \(k\geq 1\), $A \subseteq \mathbb [k]:= \{1, \dots, k\}$, and $\bs f \in B(E)^{k}$, we write
\[
  \psi_t^{(A)}[\bs f](x) = \mathbb E_{\delta_x}\left[\prod_{i \in A}X_t[f_i]\right], \quad t \ge 0.
\]
If $A = [k]$, we will write $\psi_t^{(k)}$ instead. 

Similarly to \cite{bmoments}, we will use an inductive argument to prove our results concerning the asymptotic behaviour of $\psi_t^{(k)}$. For this, we introduce the following notation.

For a set $A \subset \mathbb N$, we let $\mathcal P(A)$ denote the set of partitions of $A$ and set $\mathcal{P}^*(A) := \mathcal P(A) \setminus \{A\}$. Further, if $|A| \ge 2$ and $2 \le \ell \le |A|$, we set
\[
  B_{\ell, A} := \{\bs i = (i_1, \dots, i_\ell) \in A^{\ell}: i_j \neq i_k, j \neq k\}.
\]
For $\sigma \in \mathcal P(A)$, we write $|\sigma|$ for the number of blocks in $\sigma$ and $\sigma_j$ for the $j$-th block.
For $k \ge 1$, we write $\mathcal P(k)$, $\mathcal P^*(k)$ and $B_{\ell, k}$ in place of $\mathcal P([k])$, $\mathcal P^*([k])$ and $B_{\ell, [k]}$, respectively. Then, for $A \subset \mathbb N$ and a function $\bs g : \mathfrak P(A) \times E \to \mathbb C$, where $\mathfrak P(A)$ denotes the power set of $A$, we define
\begin{equation}
    \zeta_A[\bs g](x) = \mathcal E_x \left[\sum_{\sigma \in \mathcal P^*(A)} \sum_{\bs i \in B_{|\sigma|,N}}\prod_{j = 1}^{|\sigma|} \bs g(\sigma_j, x_{i_j})\right].
    \label{eq: zeta branching operator}
\end{equation}
The above operator will allow us to write $\psi_t^{(k)}$ in terms of the lower order moments, thus motivating inductive proofs. For similar reasons, we would like to develop asymptotics akin to that of \ref{H1b} but for the other eigenvalues. To this end, recall the notation from \ref{H1b} and define
\begin{align*}
    &\mathrm{Ei}^*(\lambda_1) = \{f \in B(E):\text{\(\Phi_{1,j}[f]\neq 0\) for some \(1\leq j \leq p_1\)}\},\\
    &\mathrm{Ei}^*(\lambda_{i}) = \{f \in B(E):\text{\(\Phi_{i,j}[f]\neq 0\) for some \(1\leq j \leq p_{i}\)}\}, \quad 2\leq i \leq m,
\end{align*}
and 
\begin{align*}
    \mathrm{Ei}(\lambda_1) =\mathrm{Ei}^*(\lambda_1), \quad \mathrm{Ei}(\lambda_{i}) =\mathrm{Ei}^*(\lambda_i) \setminus \bigcup_{\substack{\mathrm{Re}\lambda_j \geq   \mathrm{Re}\lambda_{i}\\\lambda_j\neq \lambda_{i}}}\mathrm{Ei}^*(\lambda_j), \quad 2 \leq i \leq m.
\end{align*}
We will also use the notation \(\mathrm{Ei}_1(\lambda_i)\) to denote \(\mathrm{Ei}(\lambda_i)\) restricted to functions bounded by unity. Then, thanks to \eqref{eq: h1b assum}, for \(1\leq i \leq m\), we have
\begin{equation}
    \sup_{x \in E,f \in \mathrm{Ei}_{1}(\lambda_{i})}\e^{-\mathrm{Re}\lambda_{i} t}\bigg|\psi_t[f](x) - \\e^{(\lambda_{i}+\mathcal{N})t}\Phi_{i}[f](x)\bigg|
  \to 0, \quad \text{as } t \to \infty. \label{eq: h1b assum 2}
\end{equation}

For \(1\leq i \leq m\) and \(f \in \mathrm{Ei}_1(\lambda_{i})\), if \(2\mathrm{Re}\lambda_{i}\geq \lambda_1\), then we will show that the asymptotic growth rate of the higher moments of \(X_t[f]\) depends on the growth rate of the semigroup \(\psi_t[f]\) given in \eqref{eq: h1b assum 2}. In this setting, we say that \(\lambda_{i}\) is {\it large} when $2\mathrm{Re}\lambda_{i} > \lambda_1$ and {\it critical} when $2\mathrm{Re}\lambda_{i} = \lambda_1$. If however, \(2\mathrm{Re}\lambda_{i}< \lambda_1\), then the asymptotic growth rate of the higher moments of \(X_t[f]\) depends instead on \(\lambda_1\) and \(p_1\). We say that \(\lambda_{i}\) is {\it small} when $2\mathrm{Re}\lambda_{i} < \lambda_1$. For this reason, we will work with a different function space for each of these three regimes, which we now define.

When it exists, let \(\tau\) be the unique index such that \(2\mathrm{Re\lambda}_{\tau-1}\geq \lambda_1\) but \(2\mathrm{Re\lambda}_{\tau}< \lambda_1\). Set
\begin{equation*}
    \mathrm{Ei}(\Lambda_{L}) := \bigcup_{2\mathrm{Re}\lambda_{i}>\lambda_1}\mathrm{Ei}(\lambda_{i}), \quad  \mathrm{Ei}(\Lambda_{C}) := \bigcup_{2\mathrm{Re}\lambda_{i}=\lambda_1}\mathrm{Ei}(\lambda_{i}),\quad  \mathrm{Ei}(\Lambda_{S}):= \mathrm{ker}(\lambda_{\tau-1}),
\end{equation*}
where ${\rm ker}(\lambda_{\ell})$ was defined just before \ref{H1b}, and
where if \(\tau\) is not well defined (i.e. there are no small eigenvalues), we set \(\mathrm{Ei}(\Lambda_{S})=\emptyset\). 

For \(f \in \mathrm{Ei}(\Lambda_{L}) \cup \mathrm{Ei}(\Lambda_{C})\), let \(\nu(f)\) be the unique index such that \(f \in \mathrm{Ei}(\lambda_{\nu(f)})\), set \(\lambda(f):=\lambda_{\nu(f)}\) and define
\begin{align*}
p(f):= \max \{i\geq 0: \mathcal{N}^{i-1}f \neq 0\}.
\end{align*} 
Then, for \(f \in (\mathrm{Ei}(\Lambda_{L}) \cup \mathrm{Ei}(\Lambda_{C}))^k\), define
\begin{align*}
    &\tilde \lambda(\bs f) = \sum_{i=1}^k \lambda(f), \quad \tilde p(\bs f) = \sum_{i=1}^k p(f).
\end{align*}

{Finally, for $\ell \ge 1$ with $\ell$ even and $\bs f \in \mathrm{Ei}(\Lambda_C)^{\ell}$, we say that $(\lambda( f_1), \dots, \lambda( f_{\ell}))$ is a {\it conjugate $\ell$-tuple} if one can split \((f_i)_{i=1}^{\ell}\) into $\ell/2$ pairs such that, for each pair \((f_i,f_{j})\), \(\lambda(f_i)=\bar \lambda(f_j)\).}

\section{Main results}\label{sec:main}

In this section, we will present our main results. We are interested in the asymptotic behaviour of $\psi_t^{(k)}$. As briefly mentioned in the previous section, this behaviour depends on whether we are in the large, critical or small regime. For this reason, we will only consider the asymptotics of $\psi_t^{(k)}[\bs f]$ for $\bs f \in {\rm Ei}(\Lambda_L)^k \cup {\rm Ei}(\Lambda_C)^k \cup {\rm Ei}(\Lambda_S)^k$. In the case that we have \(\bs f \in B_1(E)^k\), the same proof techniques can be used, however this would result in significantly longer theorems and proofs, since many more cases need to be considered.

From here on we assume that \ref{H1b} is in force.


\begin{theorem}[Large Regime]
\label{theorem: large regime}
 Assume that for some \(k\geq 2\), \[\sup_{x\in E}\mathcal{E}_x[N^k]<\infty.\] For \(1\leq \ell \leq k\) and $t \ge 0$ set 
\begin{equation*}
   \Delta_{\ell,t}=\sup _{x \in E, \bs f \in \mathrm{Ei}_1(\Lambda_L)^{\ell}}\left|\e^{-\tilde\lambda(\bs f) t}(1+t)^{\tilde{p}(\bs f)-\ell}\psi_t^{(\ell)}[\bs f](x)-\prod_{j=1}^{\ell}(p(f_j)-1)!L_{[\ell]}[\bs f^*](x)\right|,
\end{equation*}
where, for \(1\leq i \leq \ell\),
\(L_{\{i\}}[ \bs f](x) = \Phi_{\nu(f_i)}[f_i](x)\), and for $A \subseteq [\ell]$ with \(2 \le |A|\le \ell\),
\begin{equation}
\label{eq: el functions}
    L_A[\bs f](x) = \int^{\infty}_{0}\e^{-\tilde\lambda(\bs f) s}\psi_s\F[\gamma \zeta_{A}[L_{\cdot}[(\e^{-{\mathcal{N}s}}f_1,\dots,\e^{-{\mathcal{N}s}}f_{\ell})]]\R](x)\mathrm{d}s,
\end{equation}
and where
\begin{equation}
\label{eq: f star}
    \bs f^* = (\mathcal{N}^{p(f_1)-1}f_1,\dots,\mathcal{N}^{p(f_{\ell})-1}f_{\ell}).
\end{equation}
Then, for \(1\leq \ell \leq k\), 
\begin{equation*}
    \sup_{t \geq 0} \Delta_{\ell,t}<\infty \text { and } \lim _{t \rightarrow \infty} \Delta_{\ell,t}=0.
\end{equation*}   
\end{theorem}
\bigskip

\begin{theorem}[Small Regime] 
\label{theorem: small regime}
Assume that for some \(k\geq 2\), \[\sup_{x\in E}\mathcal{E}_x[N^k]<\infty.\] 
For $t \ge 0$ and \(2\leq \ell \leq k\), let 
\begin{equation}
   \Delta_{\ell,t}=\sup _{x \in E, \bs f \in \mathrm{Ei}_1(\Lambda_S)^{\ell}}\left|\e^{-\frac{\ell\lambda_1}{2}}(1+t)^{-\frac{\ell (p_1-1)}{2}}\psi_t^{(\ell)}[\bs f](x)-L_{[\ell]}[\bs f](x)\right|, \label{eq: small regime 2 moment 2}
\end{equation}
where $\lambda_1$ and $p_1$ were given in \ref{H1b},
and, for $A \subseteq \{1, \dots, \ell\}$ with $|A|$ even, if $|A| = 2$, we have
\begin{equation*}
    L_A[\bs f](x) = (p_1-1)!\Phi_{1,1}[\mathcal{N}^{p_1-1}(f_{A_1}f_{A_2})](x)+(p_1-1)!\int_0^{\infty}\e^{-\lambda_1 s}\Phi_{1,1}[\mathcal{N}^{p_1-1}(\gamma \zeta_A[\psi_s^{(\cdot)}[\bs f]])](x)\mathrm{d}s, 
\end{equation*}
and for $|A| \ge 4$, we have
\[
L_A[\bs f](x) = 
    \int^{\infty}_{0}\e^{-\frac{ \lambda_1 |A|}{2}}\psi_s\F[\gamma \zeta_A[L_\cdot[\bs f]]\R](x)\mathrm{d}s.
\]
Otherwise, $L_A[\bs f](x) = 0$.
Then, for \(1\leq \ell \leq k\), 
\begin{equation*}
    \sup_{t \geq 0} \Delta_{\ell,t}<\infty \text { and } \lim _{t \rightarrow \infty} \Delta_{\ell,t}=0.
\end{equation*}
\end{theorem}

\begin{remark}~ 
\label{remark : odd small moments}
 Note that under a different scaling one can use the techniques presented here to obtain convergence of the odd moments in the small regime. This scaling depends on whether \(\mathrm{Re}\lambda_{\tau}\) is greater than, equal to, or less than \(0\), and we see a similar trichotomy as in this article. We omit these results, since it would make the presentation much more technical whilst providing little gain.
\end{remark}


\bigskip

\begin{theorem}[Critical Regime]
\label{theorem: critical regime}
Assume that for some \(k\geq 2\), \[\sup_{x\in E}\mathcal{E}_x[N^k]<\infty.\] For \(2\leq \ell \leq k\), let 
\begin{equation}
   \Delta_{\ell,t}=\sup _{x \in E, \bs f \in \mathrm{Ei}_1(\Lambda_C)^{\ell}}\left|\e^{-\frac{\ell\lambda_1}{2}}(1+t)^{-\F(p(\bs f)+\frac{\ell (p_1-2)}{2}\R)}\psi_t^{(\ell)}[\bs f](x)-L_{[\ell]}[\bs f](x)\right|, \label{eq: critical regime 2 moment 2}
\end{equation}
where $\lambda_1$ and $p_1$ were given in \ref{H1b}, and, for $A \subseteq \{1, \dots, \ell\}$ with $|A|$ even, we have
\begin{equation*}
   L_A[\bs f](x)  = \frac{\Phi_{1,1}[\mathcal{N}^{p_1-1}(\gamma \zeta_{[2]}[\Phi_{\lambda(\cdot),1}[\bs f^*]])](x)(\tilde p(\bs f)-2)!}{(p(f_1)-1)!(p(f_2)-1)!(p_1+\tilde p(\bs f)-2)!}, \quad |A| = 2, \quad (f_i:i\in A)\text{ conjugate \(2\)-tuple},
\end{equation*}
where \(\bs f^* \) is as in \eqref{eq: f star}, and for $|A| \ge 4$, 
\begin{align*}
    & L_A[\bs f](x) = \int^{\infty}_{0}\e^{-\frac{ \lambda_1 |A|}{2}}\psi_s\F[\gamma \zeta_A[L_\cdot[\bs f]]\R](x)\mathrm{d}s, \quad (f_i:i\in A) \text{ conjugate \(|A|\)-tuple}.
\end{align*}
Otherwise \(L_A[\bs f]=0\).
Then, for \(1\leq \ell \leq k\), 
\begin{equation*}
    \sup_{t \geq 0} \Delta_{\ell,t}<\infty \text { and } \lim _{t \rightarrow \infty} \Delta_{\ell,t}=0.
\end{equation*}
\end{theorem}
Similarly to Remark \ref{remark : odd small moments}, we can use the techniques presented here to obtain convergence of the odd moments under a different scaling. In this setting, the conjugate tuple property extends by asking for the final eigenvalue to be real and equal to \(\lambda_1/2\).

\section{Proofs}
This section is dedicated to the proofs of the main results. The proofs will follow similar ideas to those presented in \cite{bmoments}. The main idea is to use an inductive argument using Lemma \ref{lem:evo-2}, with the base case given by \ref{H1b}. Before starting the proofs, we introduce some notation and several useful inequalities which will be used throughout. We use \(C\) for a generic non-negative constant that may change between equations. For multiple generic constants, we use the notation \(C_1,C_2,\dots .\) Firstly, for \(\bs f \in (\mathrm{Ei}(\Lambda_{L})\cup \mathrm{Ei}(\Lambda_{C}))^k\), and \(A \subseteq [k]\), we extend the definitions of $\tilde p(\bs f)$ and $\tilde\lambda(\bs f)$ as follows:
\begin{align*}
    &\tilde p(\bs f ,A) = \sum_{i \in A}p(f_i),\\
    &\tilde \lambda(\bs f ,A)=\sum_{i \in A}\lambda(f_i).
\end{align*}
Next, since \(\mathcal{N}\) is non-zero on a finite dimensional space, it is bounded and therefore there exists a constant \(C\) such that, for any \(f\in B(E)\),
\begin{equation}
\label{eq: bound on nilpot oper}
    \|\e^{\mathcal{N}t}f\|_{\infty} \leq C (1+t)^{p(f)-1}\|f\|_{\infty}, \quad t \in \mathbb{R}.
\end{equation}
This and \ref{H1b} imply the existence of a further constant \(C\) such that, for \(1\leq i \leq m\) and \(f\in\mathrm{Ei}(\lambda_{i})\), 
\begin{equation}
\label{eq: control on semigroup}
   \|\psi_t[f]\|_{\infty} \leq C \e^{\mathrm{Re}\lambda_{i} t}(1+t)^{p(f)-1}\|f\|_{\infty},
\end{equation}
where in particular we have used \eqref{eq: h1b assum 2}. Finally, for any \(k\geq 2\), we have that
\begin{equation}
    \zeta_{[k]}[1](x) = \mathcal E_x \left[\sum_{\sigma \in \mathcal P^*(A)} \sum_{\bs i \in B_{|\sigma|,N}} 1\right]\leq \mathcal E_x[N^k].
    \label{eq: zeta branching operator bound}
\end{equation}

\subsection{Proof of Theorem \ref{theorem: large regime}} 
We first prove an intermediate result that is of a similar form to Theorem \ref{theorem: large regime}.

\begin{lemma}\label{lemma: large regime}
Assume that for some \(k\geq 2\), \[\sup_{x\in E}\mathcal{E}_x[N^k]<\infty.\] For \(1\leq \ell \leq k\) and $t \ge 0$ set 
\begin{equation*}
   \tilde\Delta_{\ell,t}=\sup _{x \in E, \bs f \in \mathrm{Ei}_1(\Lambda_L)^{\ell}}\left|\e^{-\tilde\lambda(\bs f) t}\psi_t^{(\ell)}[(\e^{-{\mathcal{N}t}}f_1,\dots,\e^{-{\mathcal{N}t}}f_{\ell})](x)-L_{[\ell]}[\bs f](x)\right|,
\end{equation*}
where the $L_{[\ell]}$ were defined in Theorem \ref{theorem: large regime}. 
Then, for \(1\leq \ell \leq k\), 
\begin{equation*}
    \sup_{t \geq 0} \Delta_{\ell,t}<\infty \text { and } \lim _{t \rightarrow \infty} \Delta_{\ell,t}=0.
\end{equation*}
\end{lemma}

\begin{proof}
Let \(k\geq 2\) and assume that Theorem \ref{theorem: large regime} holds for \(\ell \le k-1\), where the case of \(k=1\) holds by \ref{H1b}. Furthermore, for ease of notation, for a function \(\bs f \in B(E)^k\), let
\begin{equation}
    \bs f_t = (f_{t,1},\dots,f_{t,k}), \quad f_{t,i}= \e^{-\mathcal{N}t}f_i, \quad 1\leq i \leq k, \quad t\in \mathbb{R}.
\end{equation} 
Firstly, by Lemma \ref{lem:evo-2}, for any \(\bs f \in B(E)^k\), \(x\in E\), and \(t\geq 0\), 
\begin{equation}
\label{eq: supercrit second evo}
   \psi^{(k)}_t[\bs f](x) =\psi_t\F[\prod_{i=1}^kf_i\R](x)+\int_0^t \psi_s\left[\gamma \zeta_{[k]}[\psi^{(\cdot)}_{t-s}[\bs f]]\right](x) \mathrm{d} s.
\end{equation}
 Since the product of any function uniformly bounded by 1 is also uniformly bounded by 1, by \eqref{eq: bound on nilpot oper} and \eqref{eq: control on semigroup}, the first term on the right-hand side satisfies
\begin{align}
\label{eq: equation lm conv 1}
    &\lim_{t\rightarrow \infty}\sup_{x \in E, \bs f \in \mathrm{Ei}_1(\Lambda_L)^k}\e^{-\mathrm{Re}\tilde \lambda(f) t
}\F|\psi_t\F[\prod_{i=1}^kf_{t,i}\R](x)\R| = 0,
\end{align}
where we have used, for any \(\bs f \in \mathrm{Ei}(\Lambda_L)^k\), \(\mathrm{Re}\tilde \lambda(\bs f) > \lambda_1\). For the integral term in \eqref{eq: supercrit second evo}, first note that by the inductive hypothesis and \eqref{eq: bound on nilpot oper}, we have  
\begin{equation*}
     \sup_{0\leq s \leq t <\infty}\sup_{\bs f \in \mathrm{Ei}(\Lambda_L)^k}\sup_{A \subseteq [k]}(1+s)^{-(\tilde p(f,A)-k)}\|\e^{-\tilde\lambda(f,A)(t-s)}\psi^{(\cdot)}_{t-s}[\bs f_t](A,\cdot)-L_{A}[\bs f_{s}]\|_{\infty}< \infty,
\end{equation*}
where we have applied the inductive hypothesis to the functions \(((1+s)^{-(p(f_i)-1)}f_{s,i})_{i\in A}\). Furthermore, again by the inductive hypothesis, the left-hand side tends to 0 as \(t-s\rightarrow \infty\). This and \eqref{eq: zeta branching operator bound} imply
\begin{equation*}
    \sup_{0\leq s \leq t <\infty}\sup_{\bs f \in \mathrm{Ei}(\Lambda_L)^k}(1+s)^{- (\tilde p(f,A)-k)}\|\e^{-\tilde\lambda(f)(t-s)}\zeta_{[k]}[\psi^{(\cdot)}_{t-s}[\bs f_t]](x)-\zeta_{[k]}[L_{\cdot}[\bs f_{-s}]]\|_{\infty}<\infty,
\end{equation*}
and that the left-hand side tends to 0 as \(t-s\rightarrow \infty\). This and \eqref{eq: control on semigroup} give 
\begin{align*}
 &\sup_{x \in E,\bs f \in \mathrm{Ei}_1(\Lambda_L)^k}\F| \int_{0}^{t} \psi_s\left[\gamma (\e^{-\tilde\lambda(f)t}\zeta_{[k]}[\psi^{(\cdot)}_{t-s}[\bs f_t]]-\e^{-\tilde \lambda(f)s}\zeta_{[k]}[L_{\cdot}[\bs f_{s}]])\right](x) \mathrm{d} s\R|  \\
 &\leq  \sup_{\bs f \in \mathrm{Ei}_1(\Lambda_L)^k}C_1\int_{0}^{t} \e^{(\lambda_1-\mathrm{Re}\tilde \lambda(f))s} (1+s)^{p_1-1}\|\e^{-\tilde\lambda(f)(t-s)}\zeta_{[k]}[\psi^{(\cdot)}_{t-s}[\bs f_t]](x)-\zeta_{[k]}[L_{\cdot}[\bs f_{s}]]\|_{\infty} \mathrm{d} s\\
 &\leq C_2\int_{0}^{t} \e^{(\lambda_1-\mathrm{Re}\tilde \lambda(f))s} (1+s)^{\tilde p(f) +p_1-(k+1)}g(t-s) \mathrm{d} s,
\end{align*}
where \(g:\mathbb{R}_{\geq 0}\rightarrow \mathbb{R}_{\geq 0}\) satisfies \(g(t)\rightarrow 0\) as \(t\rightarrow \infty\).
Therefore,
\begin{equation}
\label{eq: equation lm conv 2}
    \lim_{t\rightarrow \infty}\sup_{x \in E,\bs f \in \mathrm{Ei}_1(\Lambda_L)^k}\F| \int_{0}^{t} \psi_s\left[\gamma (\e^{-\tilde\lambda(f)t}\zeta_{[k]}[\psi^{(\cdot)}_{t-s}[\bs f_t]]-\e^{-\tilde \lambda(f)s}\zeta_{[k]}[L_{\cdot}[\bs f_{s}]])\right](x) \mathrm{d} s\R|=0.
\end{equation}
Using \eqref{eq: equation lm conv 1} and \eqref{eq: equation lm conv 2} in \eqref{eq: supercrit second evo} gives us Lemma \ref{lemma: large regime}.
\hfill$\square$

Returning to the proof of the theorem, let \(k\geq 1\) and note that for any \(\bs f \in \mathrm{Ei}(\Lambda_{L})^{k}\), we have
    \begin{equation*}
        \bs f = \e^{\mathcal{N}t}\bs f_t.
    \end{equation*}
    By definition of \(\mathcal{N}\) and \eqref{eq: bound on nilpot oper},
    \begin{align*}
        &\sup_{f \in \mathrm{Ei}_1(\Lambda_{L})}\lim_{t\rightarrow \infty}\|t^{-(p(f)-1)}\e^{\mathcal{N}t}f-(p(f)-1)!^{-1}\mathcal{N}^{p(f)-1}f\|_{\infty} = 0,\\
        &\sup_{f \in \mathrm{Ei}_1(\Lambda_{L})}\|\mathcal{N}^{p(f)-1}f\|_{\infty} \leq C\|f\|_{\infty}.
    \end{align*}
    The corollary follows from this and applying Lemma \ref{lemma: large regime} to the functions
    \begin{equation*}
        (\mathcal{N}^{p(f_1)-1}f,\dots,\mathcal{N}^{p(f_k)-1}f), \quad f \in \mathrm{Ei}_1(\Lambda_{L}).
    \end{equation*}
\end{proof}
\subsection{Proof of Theorem \ref{theorem: small regime}}
    We start by showing the case of \(k=2\). For this, we will actually prove the following stronger result.
    \begin{lemma}
    \label{lemma: second moment small lemma}
        Assume that \[\sup_{x\in E}\mathcal{E}_x[N^2]<\infty.\]For $t \ge 0$, let 
        \begin{align}
        &\Delta_{2,t}^*=\sup _{x \in E, \bs f \in \mathrm{Ei}_1(\Lambda_S)^2}\left|\e^{-(\lambda_1+\mathcal{N})t}\psi_t^{(2)}[\bs f](x)-L_{[2]}^*[\bs f](x)\right|, \label{eq: small regime 2 moment 1}
        \end{align}
        where 
        \begin{equation*}
        L_{[2]}^*[\bs f](x) = \Phi_{1}[f_1f_2](x) + \int_0^{\infty}\e^{-(\lambda_1+\mathcal{N})s}\Phi_{1}[\gamma \zeta_{[2]}[\psi^{(\cdot)}_{s}[\bs f]]](x)\mathrm{d}s.
        \end{equation*}
        Then 
        \begin{equation*}
        \sup_{t \geq 0} \Delta^*_{2,t}<\infty \text { and } \lim _{t \rightarrow \infty} \Delta^*_{2,t}=0.
        \end{equation*}
    \end{lemma}
    By Lemma \ref{lem:evo-2}, for any \(\bs f \in B(E)^k\), \(x\in E\), and \(t\geq 0\), 
\begin{equation}
\label{eq: supercrit second evo 1}
   \psi^{(2)}_t[\bs f](x) =\psi_t\F[f_1f_2\R](x)+\int_0^t \psi_{t-s}\left[\gamma \zeta_{[2]}[\psi^{(\cdot)}_{s}[\bs f]]\right](x) \mathrm{d} s.
\end{equation}
For the first term on the right-hand side, by \ref{H1b}, we have that
\begin{equation}
\label{eq: need for small 1}
    \lim_{t\rightarrow \infty}\sup_{x \in E, \bs f \in \mathrm{Ei}_1(\Lambda_S)^2 }\F|\e^{-(\lambda_1+\mathcal{N})t}\psi_t\F[f_1f_2\R](x) - \Phi_{1}[f_1f_2](x)\R| = 0.
\end{equation}
Similarly, for the integrand in \eqref{eq: supercrit second evo 1}, we have that
\begin{equation}
    \lim_{t\rightarrow \infty}\sup_{s\geq 0}\sup_{x \in E, \bs f \in \mathrm{Ei}_1(\Lambda_S)^2 }\frac{\F|\e^{-(\lambda_1+\mathcal{N})t}\psi_{t}\left[\gamma \zeta_{[2]}[\psi^{(\cdot)}_{s}[\bs f]]\right](x)-\Phi_{1}[\gamma \zeta_{[2]}[\psi^{(\cdot)}_{s}[\bs f]](x)\R|}{1+\|\gamma \zeta_{[2]}[\psi^{(\cdot)}_{s}[\bs f]\|_{\infty}} =0. \label{eq: small integral bound 1}
\end{equation}
Next, by \eqref{eq: control on semigroup} and \eqref{eq: zeta branching operator bound}, there exists \(\varepsilon>0\), such that
\begin{equation*}
   \sup_{\bs f \in \mathrm{Ei}_1(\Lambda_S)^2 } \|\gamma \zeta_{[2]}[\psi^{(\cdot)}_{s}[\bs f]]\|_{\infty}\leq C\e^{(\lambda_1-\varepsilon)t},
\end{equation*}
where we have used \(2\mathrm{Re}\lambda_{\tau}<\lambda_1\). This, \eqref{eq: small integral bound 1}, and \eqref{eq: bound on nilpot oper} imply
\begin{align*}
   &\sup_{x \in E, \bs f \in \mathrm{Ei}_1(\Lambda_S)^2 }\F|\int_0^t \e^{-(\lambda_1+\mathcal{N})t}\psi_{t-s}\left[\gamma \zeta_{[2]}[\psi^{(\cdot)}_{s}[\bs f]]\right](x)-\e^{-(\lambda_1+\mathcal{N})s}\Phi_{1}[\gamma \zeta_{[2]}[\psi^{(\cdot)}_{s}[\bs f]](x) \mathrm{d} s\R|\\
   &\leq C_1\sup_{\bs f \in \mathrm{Ei}_1(\Lambda_S)^2 }\int_0^tg(t-s) (1+s)^{p_1-1}\e^{-\lambda_1s}(1+\|\gamma \zeta_{[2]}[\psi^{(\cdot)}_{s}[\bs f]\|_{\infty})\mathrm{d}s\\
   & \leq C_2\int_0^tg(t-s) (1+s)^{p_1-1}\e^{-\varepsilon s}\mathrm{d}s,
\end{align*}
where \(g:\mathbb{R}_{\geq 0}\rightarrow \mathbb{R}_{\geq 0}\) satisfies \(g(t)\rightarrow 0\) as \(t\rightarrow \infty\), and {to exchange the integral and operator \(\mathcal{N}\), we have used that \(\mathcal{N}\) is continuous since it is bounded}. Thus, the right-hand side tends to 0 as \(t\rightarrow \infty\). Using this and \eqref{eq: need for small 1} in \eqref{eq: supercrit second evo} gives Lemma \ref{lemma: second moment small lemma}.
\hfill$\square$

We return to the proof of Theorem \ref{theorem: small regime}. To obtain the case of \(k=2\) from Lemma \ref{lemma: second moment small lemma}, we use that, for the eigenspace with eigenvalue \(\lambda_1\), the dominating term of \(e^{\mathcal{N}t}\) is \((p_1-1)!^{-1}t^{p_1-1}\mathcal{N}^{p_1-1}\). Furthermore, by \ref{H1b}, for \(f \in B(E)\),
\begin{align*}
   & \mathcal{N}^{p_1-1}\Phi_{1}[f] = \Phi_{1,1}[\mathcal{N}^{p_1-1}f].
\end{align*}
Thus, \eqref{eq: small regime 2 moment 2} for \(k=2\) follows from this and \eqref{eq: small regime 2 moment 1}. The inductive step that completes the proof follows an identical structure to the proof of Theorem \ref{theorem: large regime}. We thus omit the details.
\hfill $\square$

\subsection{Proof of Theorem \ref{theorem: critical regime}}
     We start by showing the case of \(k=2\). Similarly to the proof of Theorem \ref{theorem: small regime}, we prove the following stronger result.
     \begin{lemma}
     \label{lemma: second moment critical}
Assume that \[\sup_{x\in E}\mathcal{E}_x[N^2]<\infty.\] For \(0 \leq \alpha \leq p_1-1\), let 
\begin{equation}
   \Delta_{2,t}^{\alpha}=\sup _{x \in E, \bs f \in \mathrm{Ei}_1(\Lambda_C)^2}\left|\e^{-\frac{\ell\lambda_1}{2}}(1+t)^{-\F({\tilde p(\bs f)}+p_1-\alpha-2\R)}\mathcal{N}^{\alpha}\psi_t^{(2)}[\bs f](x)-L^{\alpha}_{[2]}[\bs f](x)\right|, \label{eq: critical regime 2 moment 2}
\end{equation}
where $\lambda_1$ and $p_1$ were given in \ref{H1b}, and,
\begin{equation*}
   L_{[2]}^{\alpha}[\bs f](x)  = \frac{\Phi_{1,1}[\mathcal{N}^{p_1-1}(\gamma \zeta_{[2]}[\Phi_{\lambda(\cdot),1}[\bs f^*]])](x)(\tilde p(\bs f)-2)!}{(p(f_1)-1)!(p(f_2)-1)!(p_1+\tilde p(\bs f)-\alpha-2)!}, \quad \bs f \text{ conjugate \(2\)-tuple},
\end{equation*}
Otherwise \(L_A^{\alpha}[\bs f]=0\).
Then, for \(1\leq \ell \leq k\), 
\begin{equation*}
    \sup_{t \geq 0} \Delta_{2,t}<\infty \text { and } \lim _{t \rightarrow \infty} \Delta_{2,t}=0.
\end{equation*}
\end{lemma}
\begin{proof}
    Fix \(0\leq \alpha \leq p_1-1\). By Lemma \ref{lem:evo-2}, for any \(\bs f \in B(E)^k\), \(x\in E\), and \(t\geq 0\), 
\begin{equation}
\label{eq: supercrit second evo 2}
   \psi^{(2)}_t[\bs f](x) =\psi_t\F[f_1f_2\R](x)+\int_0^t \psi_{t-s}\left[\gamma \zeta_{[2]}[\psi^{(\cdot)}_{s}[\bs f]]\right](x) \mathrm{d} s.
\end{equation}
For ease of notation, let \(\beta(\bs f) = {\tilde  p(\bs f)}+p_1-\alpha-2\). For the first term on the right-hand side, an identical argument to \eqref{eq: need for small 1} implies
\begin{equation*}
    \lim_{t\rightarrow \infty}\sup_{x \in E, \bs f \in \mathrm{Ei}_1(\Lambda_C)^2 }\F|(1+t)^{-\beta(\bs f)}\e^{-\lambda_1t}\mathcal{N}^{\alpha}\psi_t\F[f_1f_2\R](x)\R| = 0.
\end{equation*}
We now show convergence of the integral term in \eqref{eq: supercrit second evo 2}. Let 
\begin{equation*}
   h_{\bs f}(x)=\frac{\gamma \zeta_{[2]}[\Phi_{\lambda(\cdot),1}[\bs f^*]](x)}{(p(f_1)-1)!(p(f_2)-1)!},
\end{equation*}
where \(\bs f^*\) is as in \eqref{eq: f star}. By boundedness of \(\mathcal{N}\) and \eqref{eq: zeta branching operator bound}, there exists constant \(C\), such that
\begin{equation}
\label{eq: bound for the hf}
  \sup_{\bs f \in \mathrm{Ei}_1(\Lambda_C)^2}\|h_{\bs f}\|_{\infty}<C.  
\end{equation} This, \eqref{eq: zeta branching operator bound} and \ref{H1b} imply
\begin{equation}
\label{eq: need it crit m bound 1 }
    \lim_{t\rightarrow \infty}\sup_{x \in E, \bs f \in \mathrm{Ei}_1(\Lambda_C)^2 }\F|t^{-(\tilde p(\bs f)-2)}\e^{-\tilde \lambda(\bs f)t}\gamma \zeta_{[2]}[\psi^{(\cdot)}_{t}[\bs f]](x)-h_{\bs f}(x)\R|=0.
\end{equation}
Next, by \ref{H1b}, we have
\begin{equation}
    \lim_{t\rightarrow \infty}\sup_{s\geq 0}\sup_{x \in E, \bs f \in \mathrm{Ei}_1(\Lambda_C)^2 }\frac{\F|\e^{-\lambda_1t}\mathcal{N}^{\alpha}\psi_{t}\left[\gamma \zeta_{[2]}[\psi^{(\cdot)}_{s}[\bs f]]\right](x)-\e^{\mathcal{N}t}\mathcal{N}^{\alpha}\Phi_{1}[\gamma \zeta_{[2]}[\psi^{(\cdot)}_{s}[\bs f]](x)\R|}{1+\|\gamma \zeta_{[2]}[\psi^{(\cdot)}_{s}[\bs f]\|_{\infty}} =0. \label{eq: need it crit m bound 2}
\end{equation}
Thus, by feeding the limit \eqref{eq: need it crit m bound 1 } into \eqref{eq: need it crit m bound 2}, and using \eqref{eq: bound for the hf}, we obtain, for \(x \in E\), and \(\bs f \in \mathrm{Ei}_1(\Lambda_C)^2\),
\begin{align*}
    &(1+t)^{-\beta(\bs f)}\F|\int_0^t \e^{-\lambda_1t}\mathcal{N}^{\alpha}\psi_{t-s}\left[\gamma \zeta_{[2]}[\psi^{(\cdot)}_{s}[\bs f]]\right](x) -\e^{(\tilde \lambda(\bs f)-\lambda_1)s}{s^ {\tilde p(\bs f)-2}\e^{\mathcal{N}(t-s)}}\mathcal{N}^{\alpha}\Phi_{1}[h_{\bs f}](x)\mathrm{d} s\R|\\
    &\leq (1+t)^{-\beta(\bs f)}\F|\int_0^t \e^{-\lambda_1t}\psi_{t-s}\left[\gamma \zeta_{[2]}[\psi^{(\cdot)}_{s}[\bs f]]\right](x) -\e^{-\lambda_1s}\e^{\mathcal{N}(t-s)}\mathcal{N}^{\alpha}\Phi_{1}[\gamma \zeta_{[2]}[\psi^{(\cdot)}_{t-s}[\bs f]](x)\mathrm{d} s\R|\\
    &+(1+t)^{-\beta(\bs f)}\F|\int_0^t \e^{\mathcal{N}(t-s)}\mathcal{N}^{\alpha}\Phi_{1}[\gamma \zeta_{[2]}[\psi^{(\cdot)}_{t-s}[\bs f]](x)-\e^{(\tilde \lambda(\bs f)-\lambda_1)s}{s^ {\tilde p(\bs f)-2}\e^{\mathcal{N}(t-s)}}\mathcal{N}^{\alpha}\Phi_{1}[h_{\bs f}](x)\mathrm{d} s\R|\\
    &\leq (1+t)^{-\beta(\bs f)}\F|\int_0^t (1+t-s)^{p_1-1-\alpha}(1+s)^{\tilde p(f)-2}g_1(t-s)\mathrm{d} s\R|\\
    & + (1+t)^{-\beta (\bs f)}\int_0^t (1+t-s)^{p_1-1-\alpha}(1+s)^{\tilde p(f)-2}g_2(s)\mathrm{d}s 
\end{align*}
where \(g_1,g_2:\mathbb{R}_{\geq 0}\rightarrow \mathbb{R}_{\geq 0}\), \(C\) are independent of the choice of \(x\) and \(\bs f\), and \(g_1(t),g_2(t)\rightarrow 0\) as \(t\rightarrow \infty\). Therefore, the right-hand side tends to 0 as \(t\rightarrow \infty\). It is left to show convergence of
\begin{equation*}
    (1+t)^{-\beta(\bs f)}\int_0^t \e^{(\tilde \lambda(\bs f)-\lambda_1)s}{s^ {\tilde p(\bs f)-2}\e^{\mathcal{N}(t-s)}}\mathcal{N}^{\alpha}\Phi_{1}[h_{\bs f}](x)\mathrm{d} s \rightarrow \begin{cases}
   0, \quad \tilde \lambda(\bs f)-\lambda_1 \in \mathbb{C},\\\frac{\Phi_{1,1}[\mathcal{N}^{p_1-1}h_{\bs f}](\tilde p(\bs f)-2)!}{(p_1+\tilde p(\bs f)-2)!}, \quad  \tilde\lambda(\bs f)-\lambda_1=0.  
   \end{cases}
\end{equation*} 
as \(t\rightarrow \infty\). First note that, by \ref{H1b} and \eqref{eq: bound for the hf}, for any \(0<\varepsilon<1\),
\begin{equation*}
    \lim_{t\rightarrow \infty}\sup_{\bs f \in \mathrm{Ei}_1(\Lambda_C)^2}t^{\varepsilon}\F\|t^{-(p_1-\alpha-1)}e^{\mathcal{N}t}\mathcal{N}^{\alpha}\Phi_{1}[h_{\bs f}]- \frac{\Phi_{1,1}[\mathcal{N}^{p_1-1}h_{\bs f}]}{(p_1-\alpha-1)!}\R\|_{\infty} = 0.
\end{equation*}
Thus, uniformly, for any \(\bs f \in \mathrm{Ei}_1(\Lambda_C)^2\), \(x \in E\),
\begin{equation*}
    \lim_{t\rightarrow \infty}(1+t)^{-\beta(\bs f)}\F|\int_0^t\e^{(\tilde \lambda(\bs f)-\lambda_1)s}{s^ {\tilde p(\bs f)-2}\Bigg(\e^{\mathcal{N}(t-s)}}\sum_{j=1}^{p_1}\Phi_{1,j}[h_{\bs f}](x)-\frac{(t-s)^{p_1-\alpha-1}\Phi_{1,1}[\mathcal{N}^{p_1-1}h_{\bs f}]}{(p_1-\alpha-1)!}\Bigg)\mathrm{d} s\R|=0.
\end{equation*}
as \(t\rightarrow \infty\). Finally, by standard integration techniques, we have that
\begin{align*}
   &\lim_{t\rightarrow \infty}(1+t)^{-\beta(\bs f)} \frac{\Phi_{1,1}[\mathcal{N}^{p_1-1}h_{\bs f}]}{(p_1-\alpha-1)!}\int_0^t \e^{(\tilde \lambda(\bs f)-\lambda_1)s}(t-s)^{p_1-\alpha-1}{s^ {\tilde p(\bs f)-2}}\mathrm{ds}\\
   &=\begin{cases}
   0, \quad \tilde \lambda(\bs f)-\lambda_1 \in \mathbb{C},\\\frac{\Phi_{1,1}[\mathcal{N}^{p_1-1}h_{\bs f}](\tilde p(\bs f)-2)!}{(p_1+\tilde p(\bs f)-\alpha-2)!}, \quad  \lambda(\bs f)-\lambda_1=0.  
   \end{cases}
\end{align*}
\end{proof}
Returning to the proof of Theorem \ref{theorem: critical regime}, the case of \(k=2\) is handled by Lemma \ref{lemma: second moment critical} with \(\alpha=0\). The proof of \(k\geq 3\) follows an identical structure to the proof of Theorem \ref{theorem: large regime}. We omit the details.

\hfill$\square$

\section{Appendix}
In this section, we state an evolution equation relating moments of the form \(\prod_{i=1}^kX_t[f_i]\) for \(f_1,\dots,f_k \in B(E)\) to lower order product moments. 
\begin{lemma}\label{lem:evo-2}
  Fix \(k\geq 1\). Assume that 
  \begin{equation}
  \label{eq: off bound apendix}
    \sup_{x\in E}\mathcal{E}_x[N^k]<\infty. 
  \end{equation}
  
  Then, for any \(f_1,\dots,f_k \in B(E)\), \(x\in E\), and \(t\geq 0\), we have that
\begin{equation}\label{eq:evo-23}
  \psi^{(k)}_t[\bs f](x) =\psi_t\F[f_1 \cdots f_k\R](x)+\int_0^t \psi_s\left[\gamma \zeta_{[k]}[\psi^{(\cdot)}_{t-s}[\bs f]]\right](x) \mathrm{d} s, \quad t \geq 0,
\end{equation}
where we have used the notation introduced in Section \ref{sec:main}.
\end{lemma}

This result is an extension of \cite[Proposition 9.1]{bNTEbook}. Since the proof is extremely similar, we only sketch the steps and leave the details to the reader. 

\begin{proof}[Sketch proof]
The proof follows by induction. We start by assuming that $f_1, \dots, f_k \in B^+(E)$ with the additional assumption that the functions are real-valued. The result for $k = 1$ holds trivially since $\zeta_{[1]}[\psi^{(\cdot)}_{t}[\bs f]] = 0$ for all $t \ge 0$. Assuming the result holds for $k \ge 2$, 
for $t \geq 0, x \in E, f \in B^+(E)$ with $f$ real-valued, define
\begin{equation*}
        u_t[f](x)=\mathbb{E}_{\delta_x}\left[1-\mathrm{e}^{-X_t[f]\mathrm{d} s}\right].
    \end{equation*}
    Then, by \cite[Theorem 8.2]{bNTEbook}, for \(\theta_1,\dots,\theta_k>0\), we have that
    \begin{equation}
    \label{eq: thm 8.2 equal}
        u_t\F[\sum_{i=1}^k \theta_i f_i\R](x) =\psi_t\F[1-\e^{-\sum_{i=1}^k \theta_i f_i}\R](x)-\int_{0}^t\psi_s\F[A\F[u_{t-s}\F[\sum_{i=1}^k \theta_i f_i\R]\R]\R](x)\mathrm{d}s,
    \end{equation}
      where
    \begin{equation*}
    A[f](x)=\gamma(x) \mathcal{E}_x\left[\prod_{i=1}^N\left(1-f\left(x_i\right)\right)-1+\sum_{i=1}^N f\left(x_i\right)\right], \quad x \in E.
    \end{equation*}
Differentiating with respect to $\theta_1, \dots, \theta_k$ and setting $\theta_1 = \dots = \theta_k = 0$, as in \cite[Proposition 9.1]{bNTEbook}, yields the result for $f_1, \dots, f_k$ non-negative and real-valued. To remove the restriction that $f_1, \dots, f_k$ are non-negative, we can write each $f_i$ as the difference of its positive and negative parts. Some algebra, combinatorics and the result for non-negative functions then yields the result for general bounded, real-valued functions. Finally, to extend the result to complex-valued functions, we write each complex function as the sum of its real and imaginary parts and the proof then follows as in the real case. 
\end{proof}

\section*{Acknowledgements}
This work was supported by the EPSRC grant MaThRad EP/W026899/1. Part of this work was completed while the first author was in receipt of a scholarship from the EPSRC Centre for Doctoral Training in Statistical Applied Mathematics at Bath (SAMBa), under the project EP/S022945/1.

\bibliographystyle{plain}
\bibliography{bibbp.bib}

\begin{thebibliography}{10}

\bibitem{AH1}
S.~Asmussen and H.~Hering.
\newblock Strong limit theorems for general supercritical branching processes
  with applications to branching diffusions.
\newblock {\em Z. Wahrscheinlichkeitstheorie und Verw. Gebiete},
  36(3):195--212, 1976.

\bibitem{NTEmoments}
Eric Dumonteil, Emma Horton, Andreas~E Kyprianou, and Andrea Zola.
\newblock Limit theorems for the neutron transport equation.
\newblock {\em arXiv preprint arXiv:2407.04820}, 2024.

\bibitem{durhamCMJ}
S.D. Durham.
\newblock Limit theorems for a general critical branching process.
\newblock {\em Journal of Applied Probability}, 8(1):1--16, 1971.

\bibitem{EHK}
J.~Engl\"{a}nder, S.C. Harris, and A.~E. Kyprianou.
\newblock Strong law of large numbers for branching diffusions.
\newblock {\em Ann. Inst. Henri Poincar\'{e} Probab. Stat.}, 46(1):279--298,
  2010.

\bibitem{meatman}
J.~Fleischman.
\newblock Limiting distributions for branching random fields.
\newblock {\em Trans. Amer. Math. Soc.}, 239:353--389, 1978.

\bibitem{genealogies_felix}
F{\'e}lix Foutel-Rodier and Emmanuel Schertzer.
\newblock Convergence of genealogies through spinal decomposition with an
  application to population genetics.
\newblock {\em Probability Theory and Related Fields}, 187(3):697--751, 2023.

\bibitem{bmoments}
I.~Gonzalez, E.~Horton, and A.~E. Kyprianou.
\newblock Asymptotic moments of spatial branching processes.
\newblock {\em Probability Theory and Related Fields}, 184(3-4):805--858, 2022.

\bibitem{M2few}
Simon~C Harris, Emma Horton, Andreas~E Kyprianou, and Ellen Powell.
\newblock Many-to-few for non-local branching markov process.
\newblock {\em Electronic Journal of Probability}, 29:1--26, 2024.

\bibitem{bNTEbook}
E.~Horton and A.~E. Kyprianou.
\newblock {\em Stochastic neutron transport and non-local branching Markov
  processes}.
\newblock Probability and its Applications. Birkh\"auser, 2023.

\bibitem{pedro}
Emma Horton, Andreas~E Kyprianou, Pedro Mart{\'\i}n-Ch{\'a}vez, Ellen Powell,
  and Victor Rivero.
\newblock Stability of (sub) critical non-local spatial branching processes
  with and without immigration.
\newblock {\em arXiv preprint arXiv:2407.05472}, 2024.

\bibitem{iscoe}
I.~Iscoe.
\newblock On the supports of measure-valued critical branching {B}rownian
  motion.
\newblock {\em Ann. Probab.}, 16(1):200--221, 1988.

\bibitem{Klenke}
A.~Klenke.
\newblock Multiple scale analysis of clusters in spatial branching models.
\newblock {\em Ann. Probab.}, 25(4):1670--1711, 1997.

\end{thebibliography}

\end{document}